\documentclass[a4paper,12pt]{article}
\usepackage{amssymb,amsmath,amsthm,latexsym}
\usepackage{amsfonts}
	\usepackage{amsfonts}
\usepackage{graphicx}
\usepackage[pdftex,bookmarks,colorlinks=false]{hyperref}
\usepackage{caption}
\usepackage{subcaption}
\usepackage[hmargin=1.2in,vmargin=1.2in]{geometry}

\usepackage{fancyhdr}
\newtheorem{theorem}{Theorem}[section]

\newtheorem{definition}[theorem]{Definition}

\newtheorem{lemma} [theorem]{Lemma}

\newtheorem{problem}[theorem]{Problem}

\newtheorem{remark}[theorem]{Remark}

\title{\textbf{\sc Further Studies on the Sparing Number of Graphs}}

\author{{\bf N K Sudev} $^{{1},{\ast}}$ and {\bf K A Germina $^{2}$}
\\ \\
$^{1}${\small Department of Mathematics}\\ {\small Vidya Academy of Science \& Technology} \\ {\small  Thalakkottukara, Thrissur - 680501, Kerala, India.}\\ {\small email: {\em sudevnk@gmail.com}}
\\ \vspace{0.3cm}
$^{\ast}$ {\small Corresponding author.}
\\
$^{2}${\small Department of Mathematics} \\ {\small School of Mathematical \& Physical Sciences} \\ {\small Central University of Kerala, Kasaragod - 671316, Kerala, India.}\\ {\small email: {\em srgerminaka@gmail.com}}
}
\date{}

\begin{document}
\maketitle

\begin{abstract}
Let $\mathbb{N}_0$ denote the set of all non-negative integers and $\mathcal{P}(\mathbb{N}_0)$ be its power set. An integer additive set-indexer is an injective function $f:V(G)\to \mathcal{P}(\mathbb{N}_0)$ such that the induced function $f^+:E(G) \to \mathcal{P}(\mathbb{N}_0)$ defined by $f^+ (uv) = f(u)+ f(v)$ is also injective, where $f(u)+f(v)$ is the sum set of $f(u)$ and $f(v)$. If $f^+(uv)=k~\forall~uv\in E(G)$, then $f$ is said to be a $k$-uniform integer additive set-indexer. An integer additive set-indexer $f$ is said to be a weak integer additive set-indexer if $|f^+(uv)|=\max(|f(u)|,|f(v)|)~\forall ~ uv\in E(G)$. In this paper, we study the admissibility of weak integer additive set-indexer by certain graphs and graph operations.
\end{abstract}

\vspace{0.04in}

\section{Introduction}

For all  terms and definitions, not defined specifically in this paper, we refer to \cite{FH}, \cite{BM} and \cite{ND} and for different graph classes, we further refer to \cite{BLS} and \cite{JAG}. Unless mentioned otherwise, all graphs considered here are simple, finite and have no isolated vertices.

 For two non-empty sets $A$ and $B$, the {\em sum set} of $A$ and $B$ is denoted by  $A+B$ and is defined by $A + B = \{a+b: a \in A, b \in B\}$. Using the concepts of sum sets, an integer additive set-indexer is defined as follows.

\begin{definition}\label{D2}{\rm
\cite{GA} Let $\mathbb{N}_0$ denote the set of all non-negative integers and $\mathcal{P}(\mathbb{N}_0)$ be its power set. An {\em integer additive set-indexer} (IASI, in short) is defined as an injective function $f:V(G)\to \mathcal{P}(\mathbb{N}_0)$ such that the induced function $f^+:E(G) \to \mathcal{P}(\mathbb{N}_0)$ defined by $f^+ (uv) = f(u)+ f(v)$ is also injective. A Graph which admits an IASI is called an {\em integer additive set-indexed graph} (IASI graph).}
\end{definition}

\begin{definition}\label{D3}{\rm
\cite{GS1} The cardinality of the labeling set of an element (vertex or edge) of a graph $G$ is called the {\em set-indexing number} of that element.}
\end{definition}

\begin{definition}\label{DU}{\rm
\cite{GA} An IASI is said to be {\em $k$-uniform} if $|f^+(e)| = k$ for all $e\in E(G)$. That is, a connected graph $G$ is said to have a $k$-uniform IASI if all of its edges have the same set-indexing number $k$. In particular, we say that a graph $G$ has an {\em arbitrarily $k$-uniform IASI} if $G$ has a $k$-uniform IASI  for every positive integer $k$.}
\end{definition}

The characteristics of a special type of $k$-uniform IASI graphs, called {\em weakly $k$-uniform IASI graphs}, has been studied in \cite{GS1}. A characterisation of weak IASI graphs has been done in \cite{GS3}. The following are the major notions and results established in these papers.

\begin{lemma}\label{L-Card}
\cite{GS1} For an integer additive set-indexer $f$ of a graph $G$, we have $\max(|f(u)|, |f(v)|)\le |f^+(uv)|= |f(u)+f(v)| \le |f(u)| |f(v)|$, where $u,v\in V(G)$.
\end{lemma}

\begin{definition}{\rm
\cite{GS1} An IASI $f$ is said to be a {\em weak IASI} if $|f^+(uv)|=\max(|f(u)|,\\ |f(v)|)$ for all $u,v\in V(G)$. An IASI $f$ is said to be a {\em strong IASI} if $|f^+(uv)|= |f(u)+f(v)| \le |f(u)| |f(v)|$ for all $u,v\in V(G)$. A weak  IASI is said to be {\em weakly uniform IASI} if $|f^+(uv)|=k$, for all $u,v\in V(G)$ and for some positive integer $k$.  A graph which admits a weak IASI may be called a {\em weak IASI graph}.}
\end{definition}

The following result provides a necessary and sufficient condition for a graph $G$ to admit a weak IASI.

\begin{lemma}
\cite{GS1} A graph $G$ admits a weak IASI if and only if at least one end vertex of every edge of $G$ has a singleton set-label.
\end{lemma}

\begin{definition}{\rm
\cite{GS3} An element (a vertex or an edge) of graph which has the set-indexing number $1$ is called a {\em mono-indexed element} of that graph. The {\em sparing number} of a graph $G$ is defined to be the minimum number of mono-indexed edges required for $G$ to admit a weak IASI and is denoted by $\varphi(G)$.}
\end{definition}

\begin{theorem}\label{T-WSG}
\cite{GS3} If a graph $G$ is a weak (or weakly uniform) IASI graph, then any subgraph $H$ of $G$ is also a weak (or weakly uniform) IASI graph.
\end{theorem}

\begin{theorem}\label{T-SB1}
\cite{GS3}  If a connected graph $G$ admits a weak IASI, then $G$ is bipartite or $G$ has at least one mono-indexed edge. 
\end{theorem}

From the above theorem, it is clear that all paths, trees and even cycles admit a weak IASI. We observe that the sparing number of bipartite graphs is $0$.

\begin{theorem}\label{T-WKN}
\cite{GS3}  The complete graph $K_n$ admits a weak IASI if and only if the number of edges of $K_n$ that have set-indexing number $1$ is $\frac{1}{2}(n-1)(n-2)$. 
\end{theorem}

We can also infer that a complete graph $K_n$ admits a weak IASI if and only if at most one vertex (and hence at most $(n-1)$ edges) of $K_n$ can have a non-singleton set-label.  

\begin{theorem}\label{T-NME}
\cite{GS3}  Let $C_n$ be a cycle of length $n$ which admits a weak IASI, for a positive integer $n$. Then, $C_n$ has an odd number of mono-indexed edges when it is an odd cycle and has even number of mono-indexed edges, when it is an even cycle. 
\end{theorem}

\begin{theorem}\label{T-SNC}
\cite{GS3}  The sparing number of $C_n$ is $0$ if $n$ is an even number and is $1$ if $n$ is an odd integer. 
\end{theorem}

The admissibility of weak IASIs by the union of weak IASI graphs has been established in \cite{GS4} and hence discussed the following major results.

\begin{theorem}\label{T-WUG}
\cite{GS4} Let $G_1$ and $G_2$ be two cycles. Then, $G_1\cup C_2$ admits a weak IASI if and only if both $G_1$ and $G_2$ are weak IASI graphs.
\end{theorem}

It is proved that the weak IASIs of the two graphs together constitute a weak IASI for their union. 

\begin{theorem}\label{T-SNUG}
\cite{GS4} Let $G_1$ and $G_2$ be two weak IASI graphs. Then, $\varphi(G_1\cup G_2) = \varphi(G_1)+\varphi(G_2)-\varphi(G_1\cap G_2)$.  More over, if $G_1$ and $G_2$ are edge disjoint graphs, then $\varphi(G_1\cup G_2) = \varphi(G_1)+\varphi(G_2)$.
\end{theorem}

\section{New Results on Graph Joins}

\noindent The join of two graphs is defined as follows.
\begin{definition}{\rm
\cite{FH} Let $G_1(V_1,E_1)$ and $G_2(V_2,E_2)$ be two graphs. Then, their {\em join} (or {\em sum}), denoted by $G_1+G_2$, is the graph whose vertex set is $V_1\cup V_2$ and edge set is $E_1\cup E_2\cup E_{ij}$, where $E_{ij}=\{u_iv_j:u_i\in G_1,v_j\in G_2\}$. }
\end{definition}

Certain studies about the admissibility of weak IASIs by the join of graphs have been done in \cite{GS4}. In this section, we verify the admissibility of a weak IASI by certain graphs which are the joins of some other graphs.

A major result about the admissibility of weak IASI by graph joins, proved in \cite{GS4}, is the following.

\begin{theorem}\label{T-WG++}
\cite{GS4} Let $G_1,G_2,G_3,\ldots, G_n$ be weak IASI graphs. Then, the graph $\sum_{i=1}^{n}G_i$ is a weak IASI graph if and only if all given graphs $G_i$, except one, are $1$-uniform IASI graphs.
\end{theorem} 

Invoking Theorem \ref{T-WG++}, we study about the sparing number of certain graph classes which are the joins of some weak IASI graphs. First, recall the definition of a fan graph.

\begin{definition}{\rm
\cite{JAG} The graph $P_n+\overline{K}_m$ is called an {\em $(m,n)$-fan graph} and is denoted by $F_{m,n}$.} 
\end{definition}

\noindent The following result estimates the sparing number of a fan graph $F_{m,n}$.

\begin{theorem}\label{T-Pn-mK1}
For two non-negative integers $m,n>1$, the sparing number of an $(m,n)$- fan graph $F_{m,n}=P_n+\overline{K}_m$ is $n$, the length of the path $P_n$.
\end{theorem}
\begin{proof}
By Theorem \ref{T-WG++}, $F_{m,n}$ admits a weak IASI if and only if either $P_n$ or $\overline{K}_m$ is $1$-uniform. 

If $P_n$ is not $1$-uniform, then no vertex of $\overline{K}_m$ can have a non-singleton set-label. In this case, the number of mono-indexed edges is $m\lfloor \frac{n+1}{2} \rfloor$. 

If $P_n$ is $1$-uniform, since no two vertices in $\overline{K}_m$ are adjacent, all vertices of $\overline{K}_m$ can be labeled by non-singleton set-labels. Therefore, no edge between $P_n$ and $\overline{K}_m$ is mono-indexed in $G$. That is, in this case, the number of mono-indexed edges in $F_{m,n}$ is $n$.

Since $m>1$, we have $n<m\lfloor \frac{n+1}{2} \rfloor$. Hence, the sparing number of $F_{m,n}$ is $n$.
\end{proof}

\noindent The above theorem raises a natural question about the sparing number of a graph which is the join of a cycle and a trivial graph. Let us recall the definition of an $(m,n)$-cone.

\begin{definition}{\rm
\cite{JAG} The graph $C_n+\overline{K}_m$ is called an {\em $(m,n)$-cone}. }
\end{definition}

The following theorem establishes the sparing number of an $(m,n)$-cone $G=C_n+\overline{K}_m$.

\begin{theorem}\label{T-Cn-mK1}
For two non-negative integers $m,n>1$, the sparing number of an $(m,n)$-cone $C_n+\overline{K}_m$ is $n$.
\end{theorem}
\begin{proof}
Let $G=C_n+\overline{K}_m$. Then, by Theorem \ref{T-WG++}, $G$ admits a weak IASI if and only if either $C_n$ or $\overline{K}_m$ is $i$-uniform. If some vertex of $C_n$ has a non-singleton set-label, then all vertices in  $\overline{K}_m$ must be mono-indexed. In this case, the number of mono-indexed vertices is $m\lceil n \rceil$. If $C_n$ is $1$-uniform, then all vertices in  $\overline{K}_m$ can be labeled by distinct non-singleton sets, as they are non-adjacent among themselves. In this case, the number of mono-indexed edges is $n$. Since $m>1$, we have $n\le m\lceil n \rceil$. Therefore, $\varphi(C_n+\overline{K}_m)$ is $n$. 
\end{proof}

\noindent Now, we need to consider some graphs which is the join of more than two graphs. Consider an $(m,n)$-tent graph which is defined as follows.

\begin{definition}{\rm
A graph $C_n+K_1+\overline{K}_m$ is called an {\em $(m,n)$-tent}. A tent graph can also be considered as the join of a wheel graph $W_{n+1}=C_n+K_1$ and the trivial graph $\overline{K}_m$.}
\end{definition}

\noindent We now proceed to find out the sparing number of an $(m,n)$-tent graph.

\begin{theorem}\label{T-Wn-mK1}
For two non-negative integers $m,n>1$, the sparing number of an $(m,n)$-tent $C_n+K_1+\overline{K}_m$ is $2n$.
\end{theorem}
\begin{proof}
Let $G=C_n+K_1+\overline{K}_m$. By Theorem \ref{T-WG++}, only one among $C_n$, $K_1$ and $\overline{K}_m$ can have non-singleton set-label at a time.

If the vertex in $K_1$ has a non-singleton set label, then both $C_n$ and $\overline{K}_m$ are $1$-uniform and hence the number of mono-indexed edges in $G$ is $(m+1)n$. 

If some of the vertices in $C_n$ are not $1$-uniform, then it has (at least) $\lceil \frac{n}{2} \rceil$ mono-indexed vertices. In this case, all vertices in $K_1\cup \overline{K}_m$ are $1$-uniform. Therefore, $G$ has a minimum of $(m+1)\frac{n}{2}$ mono-indexed edges if $C_n$ is an even cycle and has a minimum of $(m+1)\frac{n+1}{2}+1$ mono-indexed edges if $C_n$ is an odd cycle.

If both $K_1$ and $C_n$ are $1$-uniform, then all vertices of $\overline{K}_m$ can be labeled by non-singleton sets. Then, the number of mono-indexed edges in $G$ is $2n$. 

Since $m$ and $n$ are positive integers, $2n$ is the minimum among these numbers of mono-indexed edges. Hence, the sparing number of an $(m,n)$-tent is $2n$.
\end{proof}

Another class of graphs, common in many literature,  are the friendship graphs defined by

\begin{definition}{\rm 
A graph $K_1+mK_2$ is called an {\em $m$-friendship graph}, where $mK_2$ is the disjoint union of $m$ copies of $K_2$. A generalised friendship graph is the graph $K_1+mP_n$ which is usually called by an {\em $(m,n)$-friendship graph}. A graph $K_1+mC_n$ is called an {\em $(m,n)$-closed friendship graph}. A graph $K_1+mK_n$ is called an {\em $(m,n)$-complete friendship graph} or a {\em windmill graph}.}
\end{definition}

\begin{theorem}
The sparing number of an $m$-friendship graph is $n$, the sparing number of an $(m,n)$-friendship graph is $m\lfloor \frac{n+1}{2}\rfloor$ , the sparing number of an $(m,n)$-closed friendship graph is $m\lceil \frac{n}{2} \rceil$ and the sparing number of an $(m,n)$-complete friendship graph is $\frac{1}{2}mn(n-1)$.
\end{theorem}
\begin{proof}
Let $G=K_1+mK_2$. If $K_1$ is not mono-indexed, then no vertex in the copies of $K_2$ can be labeled by a non-singleton set. Therefore, all edges of $mK_2$ are mono-indexed. That is, the number of mono-indexed edges in $G$ is $m$. If $K_1$ is mono-indexed, then, in each copy $K_2$, one vertex can be labeled by a singleton set and the other vertex can be labeled by a non-singleton set. Then, one of the two edges between $K_1$ and each copy of $K_2$ is mono-indexed. That is, the number of mono-indexed edges in this case is also $m$. Hence, $\varphi(K_1+mK_2)=m$.

Now, consider the graph $G=K_1+mP_n$. If $K_1$ is not mono-indexed, then no vertex in the copies of $P_n$ can be labeled by a non-singleton set. Therefore, the number of mono-indexed edges of $G$ is $mn$. If $K_1$ is mono-indexed, then the vertices in each copy of $P_n$ can be labeled alternately by distinct non-singleton sets and distinct singleton sets. Then, the number of mono-indexed vertices in each copy of$P_n$ is $\lfloor\frac{n+1}{2}\rfloor$. Also, the edges connecting $K_1$ and these mono-indexed vertices are also mono-indexed. Hence, the total number of mono-indexed edges in $G$ is $m\lfloor \frac{n+1}{2}\rfloor$. Since $m,n>1$, $m\lfloor\frac{n+1}{2}\rfloor<mn$. Therefore, the sparing number of an $(m,n)$-friendship graph is $m\lfloor\frac{n+1}{2}\rfloor$.

Next, consider the graph $G=K_1+mC_n$. If $K_1$ is not mono-indexed, then no vertex in the copies of $C_n$ can be labeled by a non-singleton set. That is, the number of mono-indexed edges of $G$ is $mn$. If $K_1$ is mono-indexed, then the vertices in each copy of $C_n$ can be labeled alternately by distinct non-singleton sets and distinct singleton sets. Then, the number of mono-indexed vertices in each copy of$C_n$ is $\lceil\frac{n}{2}\rceil$.  Hence, the total number of mono-indexed edges in $G$ is $m\lceil \frac{n}{2}\rceil$. Since $m,n>1$, $m\lceil \frac{n}{2}\rceil<mn$. Therefore, the sparing number of an $(m,n)$-closed friendship graph is $m\lceil \frac{n}{2}\rceil$.

Now, consider the graph$G=K_1+mK_n$. If $K_1$ is not mono-indexed, then no vertex in the copies of $K_n$ can be labeled by a non-singleton set. That is, all copies of $K_n$ are $1$-uniform. Therefore, the total number of mono-indexed edges in $G$ is $\frac{1}{2}mn(n-1)$. If $K_1$ is mono-indexed, then exactly one vertex of each copy of $K_n$ can have non-singleton set-label. Therefore, there exist $n-1$ edges between $K_1$ and each copy of $K_n$ which are mono-indexed. By Theorem \ref{T-WKN}, each copy of $K_n$ has $\frac{1}{2}(n-1)(n-2)$ mono-indexed edges. Therefore, the total number of mono-indexed edges is $m(n-1)+m\frac{1}{2}(n-1)(n-2)=m\frac{1}{2}n(n-1)$. Note that the number of mono-indexed edges is the same in both cases. Hence, $\varphi(G)=\frac{1}{2}mn(n-1)$.

\noindent This completes the proof.
\end{proof}

\section{Weak IASI of the Ring sum of Graphs} 

Analogous to the symmetric difference of sets, we have the definition of the symmetric difference or ring sum  of two graphs as follows.

\begin{definition}{\rm
\cite{ND} Let $G_1$ and $G_2$ be two graphs. Then the {\em ring sum} (or {\em symmetric difference}) of these graphs, denoted by $G_1\oplus G_2$, is defined as the graph with the vertex set $V_1\cup V_2$ and the edge set $E_1\oplus E_2$, leaving all isolated vertices, where $E_1\oplus E_2=(E_1\cup E_2)-(E_1\cap E_2)$.}
\end{definition}

The following theorem establishes the admissibility of a weak IASI by the ring sum of two graphs.

\begin{theorem}
The ring sum of two weak IASI graphs admits a (an induced) weak IASI.
\end{theorem}
\begin{proof}
Let $G_1$ and $G_2$ be two graphs which admit weak IASIs $f_1$ and $f_2$ respectively. Choose the functions $f_1$ and $f_2$ in such way that no set $A_i\subset \mathbb{N}_0$ is the set-label of a vertex $u_i$ in $G_1$ and a vertex $v_j$ in $G_2$ simultaneously.

Let $H_1=G_1-G_1\cap G_2$ and $H_2=G_2-G_1\cap G_2$. Then, by Theorem \ref{T-WSG}, the restriction $f'_1=f_1|_{H_1}$ is an induced weak IASI for $H_1$ and the restriction $f'_2=f_2|_{H_2}$ is an induced weak IASI for $H_2$. Also, $G_1\oplus G_2=H_1\cup H_2$.

Now, define a function $f:V(G_1\oplus G_2)\to \mathcal{P}(\mathbb{N}_0)$ such that 

\begin{equation}
f(v)=
\begin{cases}
f'_1(v) & ~~\text{if}~~ v\in V(H_1)\\
f'_2(v) & ~~\text{if}~~ v\in V(H_2).
\end{cases}
\end{equation}
Since $f'_1$ and $f'_2$ are weak IASIs of $H_1$ and $H_2$, which are edge disjoint graphs such that $H_1\cup H_2= G_1\oplus G_2$, $f$ is a weak IASI on $G_1\oplus G_2$. 
\end{proof}

The most interesting and relevant question that arises here is about the sparing number of the ring sum of two weak IASI graphs. The following theorem estimates the sparing number of the join of two weak IASI graphs.

\begin{theorem}\label{T-SNRSG}
\cite{GS3} Let $G_1$ and $G_2$ be two weak IASI graphs. Then, $\varphi(G_1\oplus G_2) = \varphi(G_1)+\varphi(G_2)-2\varphi(G_1\cap G_2)$.
\end{theorem}
\begin{proof}
Let $G_1$ and $G_2$ be two weak IASI graphs and let $H_1=G_1-G_1\cap G_2$ and $H_2=G_2-G_1\cap G_2$. Then, $G_1\oplus G_2=H_1\cup H_2$. Therefore, $\varphi(H_1)=\varphi(G_1-G_1\cap G_2)=\varphi(G_1)-\varphi(G_1\cap G_2)$. Similarly, $\varphi(H_2)=\varphi(G_1)-\varphi(G_1\cap G_2)$.

Since $H_1\cup H_2=G_\oplus G_2$ and $H_1$ and $H_2$ are edge disjoint, $\varphi(G_1\oplus G_2)=\varphi(H_1\cup H_2)$. Then, we have
\begin{eqnarray*}
\varphi(G_1\oplus G_2) & = &\varphi(H_1)+\varphi(H_2)\\
& = & \varphi(H_1)+\varphi(H_2)~~~\text{(by Theorem \ref{T-SNUG})}\\
& = & [\varphi(G_1)-\varphi(G_1\cap G_2)]+[\varphi(G_2)-\varphi(G_1\cap G_2)]\\
& = & [\varphi(G_1)+\varphi(G_2)-2\varphi(G_1\cap G_2). 
\end{eqnarray*}
This completes the proof.
\end{proof}

\begin{remark}{\rm
It is to be noted that if $G_1$ and $G_2$ are edge disjoint graphs, then their ring sum and union are the same. In this case, $\varphi(G_1\oplus G_2)=\varphi(G_1\cup G_2)$.}
\end{remark}

Now, we proceed to discuss the sparing number of certain graphs which are the joins of path and cycles.

\begin{remark}{\rm
Let $P_m$ and $P_n$ be two paths in a given graph $G$. Then, $P_m\oplus P_n$ is a path or edge disjoint union of paths or a cycle or union of edge disjoint cycles. Hence, the sparing number of $P_m\oplus P_n$ is zero if $P_m$ and $P_n$ have at most one vertex in common after the removal of all common edges. If $P_m$ and $P_n$ have two or more common vertices after the removal of all common edges, then the $P_m\oplus P_n$ is a cycle or the union of cycles. Then, the sparing number of $P_m\oplus P_n$ is the sum of the sparing numbers of all these individual cycles. }
\end{remark}

\begin{remark}{\rm
Let $P_m$ be a path and $C_n$ be a cycle in a given graph $G$. If $P_m$ and $C_n$ are edge disjoint, then  $P_m\oplus C_n=P_m\cup C_n$.  Therefore, $\varphi(P_m\oplus C_n)=\varphi(C_n)$.  If $P_m$ and $C_n$ have some edges in common, then $P_m\oplus C_n$ is a path or edge disjoint union of a cycle and a path. Hence, the sparing number of $P_m\oplus C_n$ is the sum of individual edge-disjoint subgraphs obtained after the removal of common edges of $P_m$ and $C_n$.}
\end{remark}

The following theorem establishes the admissibility of weak IASI by the ring sum of two cycles.

\begin{theorem}\label{T-WRS}
If $C_m$ and $C_n$ are two cycles which admit weak IASIs, and $C_m\oplus C_n$ be the ring sum of $C_m$ and $C_n$. Then,
\begin{enumerate}
\item if $C_m$ and $C_n$ are edge disjoint, $\varphi(C_m\oplus C_n)=\varphi(C_m)+ \varphi(C_n)$.
\item if $C_m$ and $C_n$ are of same parity, then $(C_m\oplus C_n)$ contains even number of mono-indexed edges. Moreover, $\varphi(C_m\oplus C_n)=0$.
\item if $C_m$ and $C_n$ are of different parities, then $(C_m\oplus C_n)$ contains odd number of mono-indexed edges. Moreover, $\varphi(C_m\oplus C_n)=1$.
\end{enumerate}
\end{theorem}
\begin{proof}
Let $C_m$ and $C_n$ be two cycles which admit weak IASIs. If $C_m$ and $C_n$ have no common edges then, $C_m\oplus C_n=C_m\cup C_n$. By Theorem \ref{T-WUG}, the union of two weak IASI graphs admits a weak IASI and $\varphi(C_m\oplus C_n)=\varphi(C_m)+\varphi(C_n)$. 

Next, assume that $C_m$ and $C_n$ have some common edges. Let $v_i$ and $v_j$ be the end vertices of the path common to $C_m$ and $C_n$. Let $P_r, r<m$ be the $(v_i,v_j)$-section of $C_m$ and $P_s,s<n$ be the $(v_i,v_j)$-section of $C_n$, which have no common elements other than $v_i$ and $v_j$. Hence, we have $C_m\oplus C_n=P_r\cup P_s$ is a cycle. Then, we have the following cases.

\noindent {\em Case 1:} Let $C_m$ and $C_n$ are of same parity. Then we need to consider the following subcases.

\noindent {\em subcase-1.1:} Let $C_m$ and $C_n$ are odd cycles. If $C_m$ and $C_n$ have an odd number of common edges, then both $P_r$ and $P_s$ are paths of even length. Hence, the cycle $P_r\cup P_s$ is an even cycle. Therefore, $C_m\oplus C_n$ has a weak IASI with sparing number $0$. If $C_m$ and $C_n$ have an even number of common edges, then both $P_r$ and $P_s$ are paths of odd length. Therefore, the cycle $P_r\cup P_s$ is an even cycle. Hence, by Theorem \ref{T-NME}, $C_m\oplus C_n$ has even number of mono-indexed edges and by Theorem \ref{T-SB1}, the sparing number of $C_m\oplus C_n=0$.
	
\noindent {\em subcase-1.2:} Let $C_m$ and $C_n$ are even cycles. If $C_m$ and $C_n$ have an odd number of common edges, then both $P_r$ and $P_s$ are paths of odd length. Hence, the cycle $P_r\cup P_s$ is an even cycle. Therefore, $C_m\oplus C_n$ has a weak IASI. If $C_m$ and $C_n$ have an even number of common edges, then both $P_r$ and $P_s$ are paths of even length. Hence, the cycle $P_r\cup P_s$ is an even cycle. Therefore, by Theorem \ref{T-NME}, $C_m\oplus C_n$ has even number of mono-indexed edges and by Theorem \ref{T-SB1}, the sparing number of $C_m\oplus C_n=0$.

\noindent {\em Case 2:} Let $C_m$ and $C_n$ be two cycles of different parities. Without loss of generality, assume that $C_m$ is an odd cycle and $C_n$ is an even cycle. Then, we have the following subcases.

\noindent {\em Subcase-2.1:} Let $C_m$ and $C_n$ have an odd number of common edges. Then, the path $P_r$ has even length and the path $P_s$ has odd length. Hence, the cycle $P_r\cup P_s$ is an odd cycle. Therefore, by Theorem \ref{T-NME}, $C_m\oplus C_n$ has odd number of mono-indexed edges. More over, by Theorem \ref{T-SNC}, $\varphi(C_m\oplus C_n)=1$. 

\noindent {\em Subcase-2.2:} Let $C_m$ and $C_n$ have an even number of common edges. Then, $P_r$ has odd length and $P_s$ has even length. Hence, the cycle $P_r\cup P_s$ is an odd cycle. therefore, by Theorem \ref{T-NME}, $C_m\oplus C_n$ has odd number of mono-indexed edges and by Theorem \ref{T-SNC}, $\varphi(C_m\oplus C_n)=1$. 

\noindent This completes the proof.
\end{proof}

\noindent To discuss the next result we need the following notion. 

\begin{definition}{\rm
Let $H$ be a subgraph of the given graph $G$, then the graph $G-H$ is called {\em complement} of $H$ in $G$.}
\end{definition}

\begin{theorem}
Let $G$ be a weak IASI graph and $H$ be a subgraph of $G$. Then, $\varphi(G\oplus H)=\varphi(G)-\varphi(H)$.
\end{theorem}
\begin{proof}
Let $H$ be a subgraph of the graph $G$. The complement of $H$ in $G$, $G\oplus H=G-H$, is a subgraph of $G$. Hence, as $G$ is a weak IASI graph, by Theorem \ref{T-WSG}, the restriction of this IASI to $G-H$ is a weak IASI for $G-H$. Therefore, $\varphi(G\oplus H)=\varphi(G-H)=\varphi(G)-\varphi(H)$. 
\end{proof}

\section{Conclusion}

In this paper, we have discussed the existence of weak IASI for different graph classes which are the joins of some other graphs and for the ring sum of two graphs. Certain problems in this area are still open. We have not addressed the problem of finding a weak IASI for the join of two arbitrary graphs. The uncertainty in adjacency and incidence patterns in the arbitrary graphs make these studies complex. There are several other factors such as order and size of the graphs, degree of the vertices  which influence the set-labeling of the elements of a graph. 

The problems regarding the admissibility of weak IASIs by the join of certain graph classes,  whose adjacency and incidence relations are well-known, are yet to be settled. Some of the most promising problems among them are the following.

\begin{problem}{\rm
Estimate the sparing number of the join of two regular graphs.} 
\end{problem}

\begin{problem}{\rm
Estimate the sparing number of the join of two generalised Petersen graphs.} 
\end{problem}

\begin{problem}{\rm
Estimate the sparing number of the join of two bipartite (and complete bipartite) graphs.} 
\end{problem}

\begin{problem}{\rm
Estimate the sparing number of the join of two graphs, one of which is a complete graph.} 
\end{problem}

We have already formulated some conditions for some graph classes and graph operations to admit weak IASIs. More properties and characteristics of weak IASIs, both uniform and non-uniform, are yet to be investigated. The problems of establishing the necessary and sufficient conditions for various graphs and graph classes to have certain other types IASIs are still open. A study about those IASIs which assign sets having specific properties, to the elements of a given graph is also noteworthy. All these facts highlight a wide scope for further studies in this area.

\end{document}